\documentclass{amsart}

\usepackage{mathrsfs}
\usepackage{hyperref}
\usepackage{tikz}
\usetikzlibrary{matrix,arrows}

\newtheorem{theorem}{Theorem}[section]
\newtheorem{corollary}[theorem]{Corollary}
\newtheorem{proposition}[theorem]{Proposition}
\newtheorem{definition}[theorem]{Definition}

\theoremstyle{definition}
\newtheorem{example}[theorem]{Example}
\newtheorem{remark}[theorem]{Remark}

\newcommand{\curvedLie}{\mathscr{L}}
\newcommand{\Prod}{\operatorname{Prod}}
\newcommand{\coProd}{\operatorname{coProd}}
\newcommand{\Hom}{\operatorname{Hom}}
\newcommand{\MC}{\operatorname{MC}}
\newcommand{\CE}{\operatorname{CE}}
\newcommand{\Ho}{\operatorname{Ho}}

\title{Koszul duality and homotopy theory of curved Lie algebras}
\author{James Maunder}
\email{j.maunder@lancaster.ac.uk}
\address{Department of Mathematics and Statistics\\
		Lancaster University\\
		Lancaster LA1 4YF\\
		United Kingdom}
\thanks{The author would like to thank Andrey Lazarev for many helpful and fruitful discussions, and for his corrections and remarks on the drafted versions of the present paper. The author also owes thanks to John Greenlees and an anonymous referee for many helpful suggestions and corrections.}
		

\begin{document}

\begin{abstract}
This paper introduces the category of marked curved Lie algebras with curved morphisms, equipping it with a model structure. This model structure is---when working over an algebraically closed field of characteristic zero---Quillen equivalent to a model category of pseudo-compact unital commutative differential graded algebras; extending known results regarding the Koszul duality of unital commutative differential graded algebras and differential graded Lie algebras. As an application of the theory developed within this paper, algebraic deformation theory is extended to functors over pseudo-compact, not necessarily local, commutative differential graded algebras. Further, these deformation functors are shown to be representable.
\end{abstract}

\maketitle

\section{Introduction}

Pseudo-compact unital commutative differential graded algebras are dual to cocommutative counital differential graded coalgebras, and hence arise naturally as cochain complexes of topological spaces; at least in the simply connected case. Further, they play a notable role in rational homotopy theory \cite{neisendorfer,quillen} and serve as representing objects in formal deformation problems \cite{hinich_stacks,manetti,pridham}. Therefore, when working over a field of characteristic zero, it is natural to attempt to place them in the framework of a closed model category. Quillen \cite{quillen} produced the first result of this kind---albeit under a strong assumption of connectedness. Later the connectedness assumption was removed by Hinich \cite{hinich_stacks}. The crucial difference between the approaches of Quillen and Hinich was Hinich had chosen a finer notion of weak equivalence in his model structure: he worked with filtered quasi-isomorphisms, whereas Quillen worked with quasi-isomorphisms.

Additionally, the constructions of Quillen and Hinich have a notable attribute: the model categories are Quillen equivalent to certain model structures on the category of differential graded Lie algebras. This type of equivalence is so-called \emph{Koszul duality}. Whilst the construction of Hinich generalises the construction of Quillen, it does not completely generalise to the category of all cocommutative counital differential graded coalgebras; Hinich worked with conilpotent coalgebras in loc. cit. When the ground field is algebraically closed, the construction was completely generalised to all coalgebras by Chuang, Lazarev, and Mannan \cite{chuang_laz_mannan}. The authors therein chose to work in the dual setting of pseudo-compact unital commutative differential graded algebras. One especially pleasant and useful property shown in op. cit.\ is, when working over an algebraically closed field, any pseudo-compact commutative differential graded algebra can be decomposed into a product of local pseudo-compact algebras. In fact, there exists only two types of local algebras---Hinich algebras where the unique maximal ideal is closed under the differential and acyclic algebras where any closed element is also a boundary. Hinich algebras are precisely those studied by Hinich in \cite{hinich_stacks}, hence the name.

The Koszul duality is also extended in \cite{chuang_laz_mannan}. The Koszul dual therein is the category of formal coproducts of curved Lie algebras. This construction, however, is somewhat asymmetric, and herein this category is replaced with a more natural one---namely the category of marked curved Lie algebras with curved morphisms---providing a more intuitive and symmetric description for a Koszul dual to the category of cocommutative counital differential graded coalgebras, and one that is easier to work with.

Numerous papers discuss the homotopy theory of differential graded coalgebras over different operads; for example Positselski \cite{positselski} constructed a homotopy theory for coassociative differential graded coalgebras. However, the coalgebras were assumed to be conilpotent and this construction is not known in the completely general case. Further, Positselski worked with curved objects suggesting that in more general cases when discussing a Koszul duality one side of the Quillen equivalence should be a category consisting of curved objects; a hypothesis that is strengthened by results of \cite{chuang_laz_mannan} and this paper.

The paper is organised as follows. Sections \ref{sec_formal_prod} and \ref{sec_ext_hinich_and_cat_of_pc_cdga} recall (without proof) the necessary facts concerning the category of formal products of a given category, the extended Hinich category, and the category of pseudo-compact unital commutative differential graded algebras. For more details see the original paper \cite{chuang_laz_mannan}.

In Section \ref{sec_marked_lie} the category of marked curved Lie algebras and curved morphisms is introduced. This category is similar to the category of curved Lie algebras with strict morphisms discussed in \cite{chuang_laz_mannan}. On the other hand, the morphisms are quite different and as such the category itself is quite different. For instance, in the category of curved Lie algebras with strict morphisms an object with non-zero curvature cannot be isomorphic to one with zero curvature, but using curved morphisms there is an abundance of such isomorphisms: any Lie algebra twisting by a Maurer-Cartan element gives rise a curved isomorphism (see Remark \ref{rem_MC_twist}). The category of marked curved Lie algebras is equipped with a model structure (Definition \ref{def_classes}), using the existing model structure of curved Lie algebras with strict morphisms (see \cite{chuang_laz_mannan}) together with \emph{rectification by a marked point} (Definition \ref{def_rectification}). Rectification by a marked point provides a procedure in which a strict morphism can be obtained from a curved one. It should be noted, however, this construction is not functorial.

Section \ref{sec_equiv} formulates the main result of the paper: a Quillen equivalence between the model category of marked curved Lie algebras and the model category of pseudo-compact unital commutative differential graded algebras (Theorem \ref{thm_main}). This equivalence uses a pair of adjoint functors that have their origins in the Harrison and Chevalley-Eilenberg complexes of homological algebra, found (for example) in \cite{barr,harrison} and \cite{weibel} respectively. As alluded to above, a benefit of the Koszul duality of this paper over \cite{chuang_laz_mannan} is the symmetry of the construction.

Section \ref{sec_deformation} applies the material developed in the rest of the paper to introduce a certain class of deformation functors acting over pseudo-compact unital commutative differential graded algebras. Theorem \ref{thm_rep_def} shows these deformation functors are representable in the homotopy category of pseudo-compact unital commutative differential graded algebras. Two definitions (\ref{def_def_functor} and \ref{def_def_functor2}) are given for these deformation functors; one being slightly more general than the other. The less general Definition \ref{def_def_functor2} enjoys the benefit of not requiring the knowledge of a decomposition of a pseudo-compact commutative differential graded algebra into the product of local pseudo-compact commutative differential graded algebras; this is possible via a functor of \cite{chuang_laz_mannan}. Here deformation theory is considered via the differential graded Lie algebra approach; see \cite{goldman_millson,kontsevich,manetti,schlessinger_stasheff_deformation_rational,schlessinger_stasheff_tangent_cohomology_deformation}, for example.

The notion of a \emph{Sullivan homotopy} was introduced by Sullivan \cite{sullivan} in his work on rational homotopy theory, and Appendix \ref{sec_sullivan} defines analogues in the categories of curved Lie algebras with strict morphisms and local pseudo-compact commutative differential graded algebras. These analogues serve the constructions of Section \ref{sec_deformation} by showing (when the objects considered are suitably nice) the equivalence classes of Sullivan homotopic morphisms is in bijective correspondence with the classes of morphisms in the homotopy category (Theorem \ref{thm_sullivan_homotopy}). An explicit construction of a path object in the category of curved Lie algebras with strict morphisms and similar---but more subtle---ideas in the category of local pseudo-compact commutative differential graded algebras lead to the proof.

\section*{Notation and conventions}

Throughout the paper it is assumed all commutative and Lie algebras are over a fixed algebraically closed field, $k$, of characteristic zero. Algebraic closure is necessary for some key results, but some of the statements hold in a more general setting. Unmarked tensor products are assumed to be over $k$ and algebras are assumed to be unital, unless stated otherwise. The vector space over $k$ spanned by the vectors $a,b,c,\dots$ is denoted by $\langle a,b,c,\ldots \rangle$. Similarly, the free differential graded Lie algebra over $k$ on generators $a,b,c,\ldots$ is denoted by $\operatorname{L}\langle a,b,c,\ldots \rangle$.

The following abbreviations are commonly used throughout the paper: `dgla' for differential graded Lie algebra; `cdga' for commutative differential graded algebra; `dg' for differential graded; `MC' for Maurer-Cartan; `CMC' for closed model category in the sense of \cite{quillen} (for a review of this material see \cite{dwyer_spalinski}); `LLP' for left lifting property; and `RLP' for right lifting property.

Graded objects are assumed to be $\mathbb{Z}$-graded, unless otherwise stated. For both commutative and Lie algebras this grading is in the homological sense with lower indices. Although some Lie algebras in this paper are not necessarily complexes, they possess an odd derivation often referred to as the (pre-)differential and hence resemble complexes. In the homological grading, these differentials have degree $-1$. Given any homogeneous element, $x$, of some given algebra, its degree is denoted by $|x|$. Therefore, in the homological setting a $\MC$ element is of degree $-1$ and the curvature element of a curved Lie algebra is of degree $-2$. The suspension, $\Sigma V$, of a homologically graded space is defined by $(\Sigma V)_i = V_{i-1}$. Applying the functor of linear discrete or topological duality takes homologically graded spaces to cohomoligically graded ones, and vice versa, i.e.\ $(V^*)^i=(V_i)^*$. A homologically graded space can therefore be considered equivalently as a cohomological one by setting $V_i=V^{-i}$ for each $i\in\mathbb{Z}$. Additionally, $\Sigma V^*$ is written for $\Sigma (V^*)$, and with this convention there is an isomorphism $(\Sigma V)^*\cong\Sigma^{-1} V^*$.

Many cdgas considered in this paper are pseudo-compact. A cdga is said to be pseudo-compact if it is an inverse limit of finite dimensional commutative graded algebras with continuous differential. Taking the inverse limit induces a topology and the operations of the algebra are assumed to be continuous with respect to this topology. More details on pseudo-compact objects can be found in \cite{gabriel,keller_yang,vandenbergh}.

Consider a curved Lie algebra, $(\mathfrak{g},d_\mathfrak{g},\omega )$, and a pseudo-compact cdga, $A=\varprojlim A_i$. The \emph{completed tensor product}, denoted $\hat{\otimes}$, is given by
\[
\mathfrak{g}\hat{\otimes} A=\varprojlim_i \mathfrak{g}\otimes A_i,
\]
where the tensor on the right hand side is given by the tensor product in the category of graded vector spaces. Note the adjective `completed' is dropped almost everywhere. This tensor product possesses the structure of a curved Lie algebra: the curvature is given by $\omega \hat{\otimes} 1$, the differential is given on elementary tensors by $d(x \hat{\otimes} a)=(d_\mathfrak{g} x)\hat{\otimes} a + (-1)^{|x|}x\hat{\otimes} (d_A a)$, and the bracket is given on elementary tensors by $[x\hat{\otimes} a,y\hat{\otimes} b]=[x,y]\hat{\otimes} (-1)^{|a||y|} ab$. This construction is useful in Section \ref{sec_deformation} when defining deformation functors; Definitions \ref{def_def_functor} and \ref{def_def_functor2}.

\section{The category of formal products}\label{sec_formal_prod}

The category of formal products was defined as a means to describe the category of pseudo-compact cdgas. Here the definition and some facts are recalled; for greater details and the proofs see the original paper \cite{chuang_laz_mannan}. Let $\mathcal{C}$ be a CMC.

\begin{definition}
The category of formal products in $\mathcal{C}$, denoted $\Prod (\mathcal{C})$, is the category with objects given by morphisms from indexing sets to the set of objects of the category $\mathcal{C}$. An object is denoted by $\prod_{i \in I} A_i$, where $I$ is some indexing set and for each $i\in I$ the morphism sends $i\mapsto A_i\in \mathcal{C}$. A morphism in $\Prod (\mathcal{C} )$,
\[
f\colon\prod_{i \in I} A_i \to \prod_{j\in J} B_j,
\]
is given by a morphism of sets $J\to I$ that sends $j\mapsto i_j$, and a morphism $f_j \colon A_{i_j}\to B_j$ of $\mathcal{C}$ for all $j\in J$. The morphism $f_j$ is called the $j$th component of the morphism $f$.
\end{definition}

\begin{remark}\label{rem_terminal_in_prod}
The indexing set of an object in $\Prod (\mathcal{C})$ could be empty and in this case one has the terminal object for $\Prod(\mathcal{C})$.
\end{remark}

\begin{definition}
Let $\prod_{i\in I} A_i, \prod_{j\in J} B_j \in\Prod (\mathcal{C})$, their product is given by
\[
\prod_{i\in \left(I\coprod J\right)} A_i
\]
and their coproduct is given by 
\[
\prod_{(i,j)\in I\times J} \left(A_i\coprod B_j \right).
\]
Both constructions easily extend from the binary case.
\end{definition}

\begin{proposition}
Given a morphism $f\colon\prod_{i\in I} A_i \to \prod_{j\in J} B_j$, for each $i\in I$, let $B^i$ denote the product in $\mathcal{C}$ of the $B_j$ satisfying $i_j = i$. The morphisms $f_j \colon A_i \to B_j$ factor uniquely through a morphism
$f^i \colon A_i \to B^i $. \qed
\end{proposition}

\begin{definition}\label{def_classes_prod}
A morphism $f\colon\prod_{i\in I} A_i \to \prod_{j\in J} B_j$ of $\Prod (\mathcal{C})$ is said to be a
\begin{itemize}
\item weak equivalence if, and only if, the morphism $J\to I$ induced by f is a bijection and for every $j\in J$ the morphism $f_j$ (or equivalently each for every $i\in I$ the morphism $f^i$) is a weak equivalence in $\mathcal{C}$;
\item cofibration if, and only if, for each $j\in J$ the morphism $f_j$ is a cofibration in $\mathcal{C}$;
\item fibration if, and only if, for each $i\in I$ the morphism $f^i$ is a fibration in $\mathcal{C}$.
\end{itemize}
\end{definition}

The classes of morphisms in Definition \ref{def_classes_prod} provide $\Prod (\mathcal{C})$ with the structure of a CMC; c.f \cite[Theorem 4.8]{chuang_laz_mannan}.

\begin{example}
To gain some intuition, an example of \cite{chuang_laz_mannan} is recalled. Let $\mathcal{C}$ be the category of connected topological spaces, then the category $\coProd (\mathcal{C}):=\Prod(\mathcal{C}^{op})$ is the category of all topological spaces that can be written as the disjoint union of connected spaces. In fact, in \cite{laz_markl}, the categories of finite (co)products were used to construct a disconnected rational homotopy theory.
\end{example}

\section{The extended Hinich category and pseudo-compact cdgas}\label{sec_ext_hinich_and_cat_of_pc_cdga}

As shown in \cite{chuang_laz_mannan}, the category of pseudo-compact cdgas (denoted herein $\mathcal{A}$) is equivalent to the category of formal products of local pseudo-compact cdgas. The category of local pseudo-compact cdgas is referred to as the \emph{extended Hinich category} and denoted by $\mathcal{E}$. This section does not contain any original results; for greater details and the proofs see the original paper \cite{chuang_laz_mannan}.

\begin{proposition}
Any pseudo-compact cdga is isomorphic to a direct product of local pseudo-compact cdgas. \qed
\end{proposition}

\begin{definition}
A local pseudo-compact cdga is said to be
\begin{itemize}
\item a Hinich algebra if the maximal ideal is closed under the differential;
\item an acyclic algebra if every cycle is a boundary.
\end{itemize}
\end{definition}

Every local pseudo-compact cdga is in fact one of the above two types, c.f.\ \cite{chuang_laz_mannan}. The category of Hinich algebras is referred to as the \emph{Hinich category}.

\begin{definition}
Given $A\in\mathcal{E}$ with maximal ideal $M$, the \emph{full Hinich subalgebra} is the local pseudo-compact cdga given as $A^H:=\lbrace a\in A : da\in M\rbrace$.
\end{definition}

Clearly, $A^H$ is a Hinich algebra and if $A$ is a Hinich algebra then $A^H=A$. If $A$ is acyclic then $A^H$ has codimension one in $A$.

\begin{proposition}\label{prop_no_morphisms_acyclic_to_hinich}
Given a morphism $f\colon A\to B$ of pseudo-compact cdgas, it restricts to a morphism $f^H\colon A^H \to B^H$ of Hinich algebras. Conversely, given $x\in A$ such that $x\notin A^H$, then $f(x)\notin B^H$. \qed
\end{proposition}

It follows from Proposition \ref{prop_no_morphisms_acyclic_to_hinich} there cannot exist a morphism from an acyclic algebra to a Hinich algebra and any morphism from a Hinich algebra to an acyclic algebra must factor through the full Hinich subalgebra of the codomain.

\begin{remark}\label{rem_functor_H}
A functor from the extended Hinich category to the Hinich category is defined by $A\mapsto A^H$ and $(f\colon A\to B) \mapsto (f^H \colon A^H\to B^H)$. This functor forms the right adjoint in a Quillen adjunction (the left adjoint being the inclusion), c.f.\ \cite[Proposition 3.19]{chuang_laz_mannan}.
\end{remark}

\begin{proposition}
The acyclic algebra $\Lambda=k[x]/\langle x^2 \rangle$ with differential given by $dx=1$, where $|x|=1$, is the terminal object of $\mathcal{E}$.\qed
\end{proposition}

\begin{definition}\label{def_model_for_extended}
A morphism, $f\colon A\to B$, of $\mathcal{E}$ is said to be a
\begin{itemize}
\item weak equivalence if, and only if, it is a weak equivalence in the Hinich category or any morphism of acyclic algebras;
\item fibration if, and only if, $B^H$ is contained within its image;
\item cofibration if, and only if, it is a retract of a morphism in the class consisting of the tensor products of cofibrations in the Hinich category with:
\begin{itemize}
\item the identity $k\to k$;
\item the identity $\Lambda \to \Lambda$;
\item the natural inclusion $k\hookrightarrow\Lambda $.
\end{itemize}
\end{itemize}
\end{definition}

\begin{remark}
Definition \ref{def_model_for_extended} provides an extension of the model structure given by Hinich \cite{hinich_stacks} for the Hinich category. It is in fact the unique one with this choice of weak equivalence, having all surjective morphisms being fibrations, and $\Lambda$ being cofibrant; c.f.\ \cite{chuang_laz_mannan}.
\end{remark}

Definition \ref{def_model_for_extended} provides a model structure for $\mathcal{E}$ making it a CMC (c.f.\ \cite[Theorem 3.17]{chuang_laz_mannan}) and hence provides a model structure for $\mathcal{A}$ via Definition \ref{def_classes_prod} making it a CMC, since $\mathcal{A}$ is equivalent to $\Prod (\mathcal{E})$.

In the closing of this section the following important observation is made. This observation is useful in the proof of Theorem \ref{thm_rep_def}.

\begin{proposition}\label{prop_algebra_morphisms}
There exists an isomorphism of sets
\[
\Hom_\mathcal{A} \left(\prod_{i\in I} A_i,\prod_{j\in J}B_j\right)\cong \prod_{j\in J} \coprod_{i\in I} \Hom_\mathcal{E} (A_i,B_j).
\]\qed
\end{proposition}

\section{The category of marked Lie algebras}\label{sec_marked_lie}

Here the category of marked Lie algebras is introduced and some basic properties discussed, including defining a model structure (Definition \ref{def_classes}). In Section \ref{sec_equiv}, the theory developed within this Section is used to prove the category of marked curved Lie algebras is Quillen equivalent to $\Prod(\mathcal{E})$, or equivalently $\mathcal{A}$, (Theorem \ref{thm_main}).

\subsection{Basic definitions}

Similar to \cite{chuang_laz_mannan}, this paper works with curved Lie algebras, but a key difference is the morphisms are curved: further the Lie algebras also possess a set of marked points. Despite this contrast, a lot of the results proven in \cite{chuang_laz_mannan} prove useful in this setting.

\begin{definition}\label{def_curvedLie}
A curved Lie algebra is a triple $(\mathfrak{g}, d, \omega )$ where $\mathfrak{g}$ is a graded Lie algebra, $d$ is a derivation with $|d|={-1}$, and $\omega \in \mathfrak{g}_{-2}$ such that $d\omega =0$ and $d^2 x = [\omega ,x ]$ for all $x\in\mathfrak{g}$. The element $\omega$ is known as the curvature.
\end{definition}

\begin{remark}
The derivation of a curved Lie algebra is often referred to as the differential, but this is an abuse of notation since it need not square to zero (unless the curvature is zero). Despite this, the term `differential' is used within this paper.
\end{remark}

\begin{definition}\label{def_morphism}
A curved morphism of curved Lie algebras is a pair
\[
(f, \alpha )\colon(\mathfrak{g},d_\mathfrak{g},\omega_\mathfrak{g})\to(\mathfrak{h},d_\mathfrak{h},\omega_\mathfrak{h})
\]
where $f\colon\mathfrak{g}\to \mathfrak{h}$ is a graded Lie algebra morphism and $\alpha\in \mathfrak{h}_{-1}$ such that:
\begin{itemize}
\item $d_\mathfrak{h} f(x)=f(d_\mathfrak{g} x)+ [\alpha ,f(x)]$ for all $x\in\mathfrak{g}$, and
\item $\omega_\mathfrak{h}=f(\omega_\mathfrak{g})+d_\mathfrak{h} \alpha + \frac{1}{2}[\alpha,\alpha]$.
\end{itemize}
The image of the curved morphism $(f,\alpha )$ is given by $\lbrace f(x)-\alpha : x\in\mathfrak{g}\rbrace \subseteq \mathfrak{h}$. The composition of two curved morphisms, when it exists, is defined as $(f,\alpha)\circ(g,\beta)=(f\circ g,\alpha + f(\beta))$. A morphism with $\alpha=0$ is said to be strict.
\end{definition}

A strict morphism is given by a dgla morphism such that $f(\omega_\mathfrak{g})=\omega_\mathfrak{h}$; these are exactly the morphisms considered in \cite{chuang_laz_mannan}. Therefore, the $\alpha$ part of a curved morphism can be seen to deform how the differential and curvature commute with the graded Lie algebra morphism $f$. Strict morphisms are particularly useful since there exists a process to obtain a strict morphism from a curved one, c.f.\ Definition \ref{def_rectification}.

\begin{definition}
A curved isomorphism is a curved morphism 
\[
(f,\alpha )\colon(\mathfrak{g},d_\mathfrak{g},\omega_\mathfrak{g})\to(\mathfrak{h},d_\mathfrak{h},\omega_\mathfrak{h})
\]
with an inverse curved morphism
\[
(f,\alpha)^{-1}\colon(\mathfrak{h},d_\mathfrak{h},\omega_\mathfrak{h})\to(\mathfrak{g},d_\mathfrak{g},\omega_\mathfrak{g})
\]
such that
$(f,\alpha )\circ (f,\alpha )^{-1}=(id_\mathfrak{h},0)$ and $(f,\alpha )^{-1}\circ (f,\alpha )=(id_\mathfrak{g},0)$.
\end{definition}

\begin{remark}\label{rem_MC_twist}
Observe that a curved Lie algebra with non-zero curvature may be isomorphic to one with zero curvature (i.e.\ a dgla). Take the curved isomorphism
\[
(id,\xi)\colon(\mathfrak{g},d,\omega)\to\left(\mathfrak{g},d+ad_\xi,\omega+d\xi+\frac{1}{2}[\xi,\xi]\right)
\]
which has inverse $(id,-\xi)$. The codomain has zero curvature if, and only if, $\xi$ is a MC element of $(\mathfrak{g},d,\omega)$. Some curved Lie algebras, however, do not possess any MC elements, unlike dglas where $0$ is always a MC element. These morphisms correspond to twisting by $\xi$, written as $\mathfrak{g}^\xi$. More details concerning twisting can be found in \cite{braun}, and how it is generalised to $\operatorname{L}_\infty$-algebras can be found in \cite{chuang_laz}. It should be noted that neither of the cited sources use the notion of a curved morphism.
\end{remark}

\begin{proposition}\label{prop_MC_to_MC}
Given a curved morphism $(f,\alpha )\colon(\mathfrak{g},d_\mathfrak{g},\omega_\mathfrak{g})\to(\mathfrak{h},d_\mathfrak{h},\omega_\mathfrak{h})$ and a $\MC$ element $x$ of $(\mathfrak{g},d_\mathfrak{g},\omega_\mathfrak{g})$, then $(f,\alpha)(x)=f(x)-\alpha$ is a $\MC$ element of $(\mathfrak{h},d_\mathfrak{h},\omega_\mathfrak{h})$.
\end{proposition}
\begin{proof}
It is a simple check to see $f(x)-\alpha$ satisfies the $\MC$ equation.
\end{proof}

\begin{definition}
A curved Lie algebra, $(\mathfrak{g}, d,\omega )$, is said to be marked when it is equipped with a set of possibly non-distinct elements of $\mathfrak{g}_{-1}$ indexed by a non-empty set $I$, i.e.\ a set $\lbrace x_i \rbrace_{i\in I}$, with $|x_i|=-1$ for all $i\in I$.
\end{definition}

The marked points are of degree $-1$ and thus could be MC elements of the curved Lie algebra. Hence twisting by such an element results in the Lie algebra having zero curvature as shown in Remark \ref{rem_MC_twist}.

For the sake of brevity, a marked curved Lie algebra is here onwards denoted by its underlying graded Lie algebra and set of marked points, omitting the differential and curvature.

\begin{definition}\label{def_marked_morphism}
A morphism of marked curved Lie algebras is a curved Lie algebra morphism $(f,\alpha)\colon(\mathfrak{g},\lbrace
x_i \rbrace_{i\in I} )\to (\mathfrak{h},\lbrace y_j\rbrace_{j\in J})$ which induces a well defined morphism of sets $\lbrace x_i \rbrace_{i\in I}\to \lbrace y_j \rbrace_{j\in J}$. The morphism $(f,\alpha )$ is an isomorphism of marked curved Lie algebras if $(f,\alpha )$ is an isomorphism of curved Lie algebras and the induced morphism on sets is a bijection.
\end{definition}

\begin{definition}
The category of marked curved Lie algebras, denoted $\curvedLie$, is the category whose objects are marked curved Lie algebras and morphisms are those of Definition \ref{def_marked_morphism}.
\end{definition}

\begin{definition}
A marked curved Lie algebra $((\mathfrak{h}, d,\omega ),\lbrace x_i \rbrace_{i\in J})$ is a marked curved Lie subalgebra of $((\mathfrak{g}, d, \omega ),\lbrace x_i \rbrace_{i\in I})$ if $\mathfrak{h}\subseteq\mathfrak{g}$ as graded Lie algebras and $J\subseteq I$ as sets.
\end{definition}

\begin{proposition}\label{prop_rect}
Given a curved Lie algebra morphism 
\[
(f,\alpha )\colon(\mathfrak{g}, d_\mathfrak{g},\omega_\mathfrak{g})\to(\mathfrak{h},d_\mathfrak{h},\omega_\mathfrak{h}),
\]
taking $x\in \mathfrak{g}_{-1}$ and denoting $(f,\alpha)(x)=y$, a strict morphism is given by
\[
(id_\mathfrak{h},y)\circ (f,\alpha)\circ(id_\mathfrak{g},-x)\colon \mathfrak{g}^{x}\to\mathfrak{h}^{y}.
\]
\end{proposition}
\begin{proof}
$(id_\mathfrak{h},y)\circ (f,\alpha)\circ(id_\mathfrak{g},-x)=(f,y+\alpha-f(x))=(f,0).$
\end{proof}

Specialising Proposition \ref{prop_rect} to marked curved Lie algebras, recall that given 
\[
(f,\alpha)\colon (\mathfrak{g},\lbrace
x_i \rbrace_{i\in I} )\to (\mathfrak{h},\lbrace y_j\rbrace_{j\in J}),
\]
for all $i\in I$ there exists $j_i \in J$ such that $(f,\alpha )(x_i)=y_{j_i}$.

\begin{definition}\label{def_rectification}
Given a morphism $(f,\alpha )\colon (\mathfrak{g},\lbrace x_i \rbrace_{i\in I})\to(\mathfrak{h},\lbrace y_j \rbrace_{j\in J})$ and fixing some $k\in I$, call the morphism $\mathfrak{g}^{x_k}\to\mathfrak{h}^{y_k}$ obtained analogously to Proposition \ref{prop_rect} the \emph{rectification by the marked point $x_k$}.
\end{definition}

\begin{proposition}\label{prop_rect_square}
Any commutative diagram of marked curved Lie algebras with an initial vertex (one such that there exists no morphisms into it, and there exists a unique morphism, up to commutativity, to every other vertex) is isomorphic to a commutative diagram with strict morphisms.
\end{proposition}
\begin{proof}
Choosing a marked point in the initial vertex induces a choice of marked point at every other vertex by images. Rectifying each morphism by these marked points completes the proof.
\end{proof}

Proposition \ref{prop_rect_square} implies that when considering decompositions of morphisms, commutative squares, and lifting problems it is sufficient to consider strict morphisms.

\subsection{Small limits and colimits}\label{sec_lim}

The category $\curvedLie$ does not possess an initial object, and, therefore, an initial object is formally added to $\curvedLie$. The resulting category is denoted by $\curvedLie_*$. It is assumed this initial object has no marked points to echo the terminal object of $\Prod(\mathcal{E})$ having no components, see Remark \ref{rem_terminal_in_prod}. The zero curved Lie algebra $((0,0,0),\lbrace 0 \rbrace)$ is the terminal object for $\curvedLie_*$.

\begin{proposition}
The product in $\curvedLie_*$ is given by the Cartesian product of the underlying graded Lie algebras, the differential is given by specialising to each component, the curvature is given by the Cartesian product of the curvature of each component, and the Cartesian product of the sets of marked points. The projection morphisms are those morphisms projecting onto each component by the identity.
\end{proposition}
\begin{proof}
It is a straightforward check.
\end{proof}

\begin{proposition}
The equaliser of two parallel morphisms
\[
(f,\alpha),(g,\beta)\colon(\mathfrak{g},\lbrace x_i \rbrace_{i\in I} )\to (\mathfrak{h}, \lbrace y_j \rbrace_{j\in J})
\]
is the largest curved Lie subalgebra $(\mathfrak{x},\lbrace x_i \rbrace_{i\in I^\prime})\subseteq(\mathfrak{g},\lbrace x_i \rbrace_{i\in I})$ upon which the two morphisms agree, with the (strict) inclusion morphism $(\mathfrak{x},\lbrace x_i \rbrace_{i\in I^\prime})\hookrightarrow(\mathfrak{g},\lbrace x_i \rbrace_{i\in I})$.
\end{proposition}
\begin{proof}
An exercise in chasing the definitions.
\end{proof}

\begin{proposition}\label{prop_coprod}
The coproduct in the category of marked curved Lie algebras is easiest to describe in the binary case: let $(\mathfrak{g},\lbrace x_i \rbrace_{i\in I})$ and $(\mathfrak{h},\lbrace y_j \rbrace_{j\in J})$ be marked curved Lie algebras, the coproduct $(\mathfrak{g},\lbrace x_i \rbrace_{i\in I})\coprod(\mathfrak{h},\lbrace y_j \rbrace_{j\in J})$ has as its underlying graded Lie algebra the free Lie algebra on $\mathfrak{g}$, $\mathfrak{h}$, and a formal element $z$ of degree $-1$. The differential is given by the rules: $d|_\mathfrak{g} = d_{\mathfrak{g}}$, $d|_\mathfrak{h} = d_\mathfrak{h} - ad_z$ and $dz=\omega_\mathfrak{h} - \omega_\mathfrak{g} - \frac{1}{2}[z,z]$. The set of marked points is given by the union of sets $\lbrace x_i \rbrace_{i\in I} \cup \lbrace y_j + z \rbrace_{j\in J}$. The resulting space has curvature equal to that of $\mathfrak{g}$. The two inclusion morphisms are given by
\[
(id_\mathfrak{g},0)\colon (\mathfrak{g},\lbrace x_i \rbrace_{i\in I}) \hookrightarrow (\mathfrak{g},\lbrace x_i \rbrace_{i\in I})\coprod(\mathfrak{h},\lbrace y_j \rbrace_{j\in J})
\]
and
\[
(id_\mathfrak{h},-z)\colon (\mathfrak{h},\lbrace y_j \rbrace_{j\in J}) \hookrightarrow (\mathfrak{g},\lbrace x_i \rbrace_{i\in I})\coprod(\mathfrak{h},\lbrace y_j \rbrace_{j\in J}).
\]
\end{proposition}
\begin{proof}
It is clear $(\mathfrak{g},\lbrace x_i \rbrace_{i\in I})\coprod(\mathfrak{h},\lbrace y_j \rbrace_{j\in J})$ is a well defined marked curved Lie algebra. Given a marked curved Lie algebra $(\mathfrak{a}, \lbrace z_k \rbrace_{k\in K})$ and morphisms
\[
(f_\mathfrak{g},\alpha):(\mathfrak{g},\lbrace x_i \rbrace_{i\in I})\to (\mathfrak{a}, \lbrace z_k \rbrace_{k\in K})
\] and 
\[
(f_\mathfrak{h},\beta):(\mathfrak{h},\lbrace y_j \rbrace_{j\in J})\to (\mathfrak{a},\lbrace z_k \rbrace_{k\in K}),
\]
define $(f,\alpha):(\mathfrak{g},\lbrace x_i \rbrace_{i\in I})\coprod (\mathfrak{h},\lbrace y_j \rbrace_{j\in J})\to (\mathfrak{a},\lbrace z_k \rbrace_{k\in K})$ by $f|_\mathfrak{g}=f_\mathfrak{g}$, $f|_\mathfrak{h}=f_\mathfrak{h}$, and $f(z)=\alpha -\beta$. Clearly $(f,\alpha )$ is a well defined morphism and the diagram
\begin{center}
\begin{tikzpicture}
\matrix (m) [matrix of math nodes, column sep=3.5em, row sep=3.5em]{
 & (\mathfrak{a},\lbrace z_k \rbrace_{k\in K}) & \\
(\mathfrak{g},\lbrace x_i \rbrace_{i\in I}) & (\mathfrak{g},\lbrace x_i \rbrace_{i\in I})\coprod (\mathfrak{h},\lbrace y_j \rbrace_{j\in J}) & (\mathfrak{h},\lbrace y_j \rbrace_{j\in J}) \\};
\path[->]
(m-2-1) edge node[above,xshift=-0.5cm] {$(f_\mathfrak{g},\alpha)$} (m-1-2)
(m-2-2) edge node[left,yshift=-0.3cm] {$(f,\alpha)$} (m-1-2)
(m-2-3) edge node[above,xshift=0.5cm] {$(f_\mathfrak{h},\beta )$} (m-1-2)
(m-2-1) edge (m-2-2)
(m-2-3) edge (m-2-2);
\end{tikzpicture}
\end{center}
commutes. Uniqueness of this construction is a quick check.
\end{proof}

The coproduct given in Proposition \ref{prop_coprod} is similar to the disjoint product of \cite{laz_markl}: it can be informally thought of as taking the disjoint union of the two marked curved Lie algebras, adding a formal MC element, and then twisting the copy of $\mathfrak{h}$ with this formal element to flatten its curvature.

\begin{proposition}
The coequaliser of two parallel morphisms
\[
(f,\alpha), (g,\beta)\colon(\mathfrak{g},\lbrace x_i \rbrace_{i\in I})\to(\mathfrak{h},\lbrace y_j \rbrace_{j\in J})
\]
is the quotient of $\mathfrak{h}$ by the ideal generated by $f(x)-g(x)$ and $\alpha -\beta$, for all $x\in\mathfrak{g}$ with the set of marked points being the quotient in a similar manner.
\end{proposition}
\begin{proof}
A painless check.
\end{proof}

\begin{proposition}
The category $\curvedLie_*$ has all small limits and colimits.
\end{proposition}
\begin{proof}
$\curvedLie_*$ contains an initial object, terminal object, products, equalisers, coproducts, and coequalisers, this is sufficient c.f.\ \cite{cat_for_work}.
\end{proof}

\subsection{Model structure}

The category of curved Lie algebras with strict morphisms, denoted $\mathcal{G}$, plays an important role in defining a model structure for the category $\curvedLie_*$, therefore the model structure for $\mathcal{G}$ given in \cite{chuang_laz_mannan} is recalled.

\begin{definition}\label{def_classes_strict}
A morphism $f\colon\mathfrak{g}\to\mathfrak{h}$ of $\mathcal{G}$ is said to be
\begin{itemize}
\item a weak equivalence if, and only if, $f$ is either a quasi-isomorphism of dglas or any morphism between curved Lie algebras that have non-zero curvature.
\item a fibration if, and only if, it is surjective;
\item a cofibration if, and only if, it has the LLP with respect to
acyclic fibrations.
\end{itemize}
\end{definition}

Recall Definition \ref{def_rectification}: given a morphism $(f,\alpha)\colon(\mathfrak{g},\lbrace x_i \rbrace_{i\in I})\to (\mathfrak{h},\lbrace y_j \rbrace_{j\in J})$ such that $(f,\alpha )(x_k)=y_{j_k}$ for some $k\in I$, the rectification by the marked point $x_k$ gives a strict morphism $(f,0)\colon\mathfrak{g}^{x_k}\to\mathfrak{h}^{y_{j_k}}$. This latter morphism can clearly be considered as a morphism in the category $\mathcal{G}$ by forgetting marked points.

\begin{definition}\label{def_classes}
Excluding those containing the initial object, a morphism
\[
(f,\alpha)\colon(\mathfrak{g},\lbrace x_i \rbrace_{i\in I})\to (\mathfrak{h},\lbrace y_j \rbrace_{j\in J})
\]
of $\curvedLie_*$ is said to be
\begin{itemize}
\item a weak equivalence if, and only if, it induces a bijection of the marked points and the rectification by each marked point is a weak equivalence of $\mathcal{G}$;
\item a fibration if, and only if, the graded Lie algebra morphism $f$ is surjective;
\item a cofibration if, and only if, it has the LLP with respect to all acyclic fibrations.
\end{itemize}
The unique morphism from the initial object to itself belongs to all three classes. The unique morphism from the initial object to any given object of $\curvedLie_*$ is a cofibration.
\end{definition}

It is clear that each of the classes of morphism given in Definition \ref{def_classes} is closed under taking retracts, and they are also closed under rectification.

\begin{proposition}\label{prop_preserved_classes}
Excluding those containing the initial object, a morphism of $\curvedLie_*$ belongs to any one of the classes of Definition \ref{def_classes} if, and only if, the rectification by every marked point belongs to the same class.
\end{proposition}
\begin{proof}
It is straightforward to show the statement holds for weak equivalences and fibrations. Any cofibration has the LLP with respect to any acyclic fibration, and since a rectification can be viewed as a retract of the original morphism it also has the LLP with respect to any acyclic fibration.
\end{proof}

\begin{proposition}\label{prop_preserved_classes_in_G}
Excluding those containing the initial object, a morphism of $\curvedLie_*$ belongs to a class in Definition \ref{def_classes} if, and only if, the rectification by each marked point of the domain considered as a morphism of $\mathcal{G}$ belongs to the same class in Definition \ref{def_classes_strict}.
\end{proposition}
\begin{proof}
This is vacuously true for weak equivalences and almost as easily seen for fibrations. Taking the rectification of a cofibration and viewing it as a morphism of $\mathcal{G}$, it must have the LLP with respect to all acyclic fibrations of $\mathcal{G}$ since it has the LLP with respect to all acyclic fibrations of $\curvedLie_*$, and in particular those strict morphisms where the sets of marked points are assumed to be the singleton set $\lbrace 0\rbrace$, i.e.\ where the morphism is already of the form of one in $\mathcal{G}$. The converse statement is almost analogous.
\end{proof}

\begin{proposition}
Given two composable weak equivalences, $(f,\alpha)$ and $(g,\beta )$, if any two of $(f,\alpha)$, $(g,\beta )$ and $(g\circ f, \beta + g(\alpha ))$ are weak equivalences then so is the third.
\end{proposition}
\begin{proof}
When all objects involved are the initial object, the result is clear. Now excluding the initial object: if any two of the morphisms induce bijections upon the sets of marked points then so must the third. Further, if any two of them are weak equivalences of $\mathcal{G}$ after rectification then again so must the third since the two of three property holds for $\mathcal{G}$.
\end{proof}

By definition, any cofibration has the LLP with respect to all acyclic fibrations, accordingly it remains to show that any acyclic cofibration has the LLP with respect to all fibrations.

\begin{proposition}\label{prop_acy_cofib_lift_wrt_fibs}
The acyclic cofibrations are precisely the morphisms that have the LLP with respect to the fibrations.
\end{proposition}
\begin{proof}
For morphisms not containing the initial object, this follows from Proposition \ref{prop_preserved_classes}, Proposition \ref{prop_preserved_classes_in_G}, and that the statement holds in the CMC $\mathcal{G}$. The only remaining case is the acyclic cofibration given by the initial object uniquely mapping to itself, where the statement is clear.
\end{proof}

In order to show a morphism of $\curvedLie_*$ can always be factorised as the composition of an acyclic cofibration followed by a fibration it is necessary to introduce the disk of a marked curved Lie algebra.

\begin{definition}
The disk of the marked curved Lie algebra $((\mathfrak{g},d_\mathfrak{g},\omega_\mathfrak{g} ),\lbrace x_i \rbrace_{i\in I})$ is the marked curved Lie algebra $((D_\mathfrak{g},\bar{d},\bar{\omega}),\lbrace u_{x_i}\rbrace_{i\in I} )$ where $D_\mathfrak{g}=\operatorname{L}\langle\bar{\omega}, u_g, v_g : g\in\mathfrak{g}_{-} \rangle $ subject to the relations $\bar{d}u_g =v_g$ and $\bar{d}v_g=[\bar{\omega} ,u_g ]$, where $\mathfrak{g}_{-}$ denotes the homogeneous elements of $\mathfrak{g}$ and $|u_g|=|g|$ for all $g\in\mathfrak{g}_{-}$. In particular, $\bar{\omega}$ is the curvature.
\end{definition}

This can informally be thought of as attaching cells to the curved Lie algebra to make it acyclic, despite curved Lie algebras not necessarily being complexes. Even in the case of a marked curved Lie algebra with zero curvature (i.e.\ a dgla with a set of marked points) the disk construction does not lead to an acyclic complex since a non-zero curvature element always exists as a generator of the disk.

\begin{remark}
The canonical strict morphism $(D_\mathfrak{g}, \lbrace u_{x_i} \rbrace_{i\in I}) \to (\mathfrak{g},\lbrace x_i \rbrace_{i\in I} )$ given by $u_g \mapsto g$, $v_g \mapsto d_\mathfrak{g} g$, and $\bar{\omega} \mapsto \omega_\mathfrak{g}$ for all $g\in\mathfrak{g}_{-}$ is a fibration.
\end{remark}

\begin{proposition}
An acyclic cofibration is given by the canonical strict morphism 
\[
\left(\operatorname{L}\langle \bar{\omega}, u_{g_i},v_{g_i} \rangle,\lbrace u_{g_i} \rbrace_{i\in I} \right)\to (D_\mathfrak{g},\lbrace g_i \rbrace_{i\in I} ).
\]
\end{proposition}
\begin{proof}
When $(\mathfrak{g},\lbrace x_i \rbrace_{i\in I})$ is the initial object the result is trivial. Let
\[
(g,0): (\mathfrak{x},\lbrace x_j \rbrace_{j\in J}) \to (\mathfrak{y},\lbrace y_k \rbrace_{k\in K})
\]
be a fibration and
\begin{center}
\begin{tikzpicture}
\matrix (m) [matrix of math nodes, row sep=2em, column sep=2em]{
(\operatorname{L}\langle\bar{\omega}, u_{g_i}, v_{g_i} : i\in I \rangle,\lbrace u_{g_i} \rbrace_{i\in I}) & (\mathfrak{x},\lbrace x_j \rbrace_{j\in J}) \\
(D_\mathfrak{g},\lbrace u_{g_i} \rbrace_{i\in I}) & (\mathfrak{y},\lbrace y_k \rbrace_{k\in K}) \\};
\path[->]
(m-1-1) edge node[above] {$(f,0)$} (m-1-2)
(m-1-1) edge (m-2-1)
(m-2-1) edge (m-2-2)
(m-1-2) edge node[right] {$(g,0)$} (m-2-2);
\end{tikzpicture}
\end{center}
be a commutative diagram. The diagram has been assumed to be strict by Proposition \ref{prop_rect_square}. A lift $(h,0 ):(D_\mathfrak{g},\lbrace g_i \rbrace_{i\in I})\to(\mathfrak{x},\lbrace x_j \rbrace_{j\in J})$ is defined as follows. For all  $\bar{\omega}, u_{g_i},v_{g_i}\in D_\mathfrak{g}$, $(h,0)$ has the same action as $(f,0)$. For all $u_g\in D_\mathfrak{g}$ that are not in $\left(\operatorname{L}\langle \bar{\omega}, u_{g_i},v_{g_i} \rangle,\lbrace u_{g_i} \rbrace_{i\in I} \right)$ there exists some $y \in\mathfrak{y}$ such that $u_g\mapsto y$ and since $(g,0 )$ is surjective one can choose $x\in\mathfrak{x}$ such that $x\mapsto y$. Thus letting $h(u_g)=x$ completely defines a lift.
\end{proof}

\begin{proposition}
Given a morphism $(\mathfrak{g},\lbrace x_i \rbrace_{i\in I})\to(\mathfrak{h},\lbrace y_j \rbrace_{j\in J})$ of $\curvedLie_*$ it can be factorised as the composition of an acyclic cofibration followed by a fibration.
\end{proposition}
\begin{proof}
Consider the pushout
\begin{center}
\begin{tikzpicture}
\matrix (m) [matrix of math nodes, row sep=2em, column sep=2em]{
(\operatorname{L}\langle\bar{\omega}, u_{y_j}, v_{y_j} : j\in J \rangle,\lbrace g_j \rbrace_{j\in J}) & (\mathfrak{g},\lbrace x_i \rbrace_{i\in I}) \\
(D_\mathfrak{h}, \lbrace y_j \rbrace_{j\in J} ) & \left(D_\mathfrak{h}\coprod_{\operatorname{L}\langle\bar{\omega}, u_{y_j}, v_{y_j} : j\in J\rangle} \mathfrak{g},\lbrace z_i \rbrace_{i\in I}\right). \\};
\path[->]
(m-1-1) edge (m-1-2)
(m-1-1) edge (m-2-1)
(m-2-1) edge (m-2-2)
(m-1-2) edge (m-2-2);
\end{tikzpicture}
\end{center}
Since it is the pushout of an acyclic cofibration the right hand morphism is also an acyclic cofibration. Furthermore, the universal property of a pushout provides a morphism $\left(D_\mathfrak{h}\coprod_{\operatorname{L}\langle\bar{\omega}, u_{y_j}, v_{y_j} : j\in J\rangle} \mathfrak{g},\lbrace z_i \rbrace_{i\in I}\right)\to (\mathfrak{h},\lbrace y_j \rbrace_{j\in J})$ that is surjective and such that the composition
\[
(\mathfrak{g},\lbrace x_i \rbrace_{i\in })\to\left(D_\mathfrak{h}\coprod\nolimits_{\operatorname{L}\langle\bar{\omega}, u_{y_j}, v_{y_j} : j\in J\rangle} \mathfrak{g},\lbrace z_i \rbrace_{i\in I}\right)\to(\mathfrak{h},\lbrace y_j \rbrace_{j\in J})
\]
is equal to the original morphism.
\end{proof}

\begin{proposition}
Any morphism in $\curvedLie_*$ can be factorised into a cofibration followed by an acyclic fibration.
\end{proposition}
\begin{proof}
In the case $(\mathfrak{g},\lbrace x_i \rbrace_{i\in I})$ is the initial object, take the factorisation
\[
(\mathfrak{g},\lbrace x_i \rbrace_{i\in I})\to(\mathfrak{h},\lbrace y_j \rbrace_{j\in J})\to(\mathfrak{h},\lbrace y_j \rbrace_{j\in J}),
\]
where the morphism $(\mathfrak{h},\lbrace y_j \rbrace_{j\in J})\to(\mathfrak{h},\lbrace y_j \rbrace_{j\in J})$ is the identity.

Assume, $(\mathfrak{g},\lbrace x_i \rbrace_{i\in I})$ is not the initial object. Let
\[
(f,\alpha )\colon (\mathfrak{g},\lbrace x_i \rbrace_{i\in I})\to (\mathfrak{h},\lbrace y_j \rbrace_{j\in J})
\]
be any morphism in $\curvedLie_*$, and for every $k\in I$ let $(f,\alpha)(x_k) = y_{j_k}$. Consider the strict morphisms
\[
(f,0)\colon(\mathfrak{g}^{x_k}, \lbrace x_i - x_k \rbrace_{i\in I})\to(\mathfrak{h}^{y_{j_k}},\lbrace y_j - y_{j_k} \rbrace_{j\in J})
\]
(obtained by rectification c.f.\ Definition \ref{def_rectification}) as morphisms in $\mathcal{G}$ by forgetting the marked points. Therefore, there exists the factorisation in $\mathcal{G}$
\begin{center}
\begin{tikzpicture}
\matrix (m) [matrix of math nodes, row sep=2em, column sep=2em]{
& \mathfrak{x} & \\
\mathfrak{g}^{x_k} & & \mathfrak{h}^{y_{j_k}}, \\
};
\path[->]
(m-2-1) edge node[below] {$(f,0)$} (m-2-3)
(m-2-1) edge node[left] {$(\iota ,0)$} (m-1-2)
(m-1-2) edge node[right] {$(p,0)$} (m-2-3);
\end{tikzpicture}
\end{center}
where $(\iota,0)$ and $(p,0)$ are a cofibration and acyclic fibration of $\mathcal{G}$ resp. Since $(f,0)=(p\circ\iota, 0)$, the image of $\lbrace x_i -x_k \rbrace_{i\in I}$ under $(\iota ,0)$ must be a subset of the pre-image $p^{-1}\lbrace y_j -y_{j_k}\rbrace_{j\in J}$. Accordingly, for each element $y_n \in\lbrace y_j -y_{j_k}\rbrace_{j\in J}$, take $z_n$ to be an element in the pre-image of $y_n$ under $(p,0)$ (which exists by surjectivity of $p$) that is either also in $(\iota ,0) \lbrace x_i -x_n \rbrace_{i\in I}$ if it exists, or any arbitrary element of correct degree if it does not. Thus, $(\mathfrak{x},\lbrace z_j \rbrace_{j\in J})$ is a marked curved Lie algebra. Untwisting $\mathfrak{g}$ and $\mathfrak{h}$ results in a decomposition of the morphism $(f,\alpha )= (p,-y_{j_k})\circ (\iota,\iota (x_k))$. Clearly $(p,-y_{j_k})$ is an acyclic fibration. To see that $(\iota,\iota (x_k))$ is a cofibration, consult the following diagram
\begin{center}
\begin{tikzpicture}
\matrix (m) [matrix of math nodes, row sep=2em, column sep=2em]{
(\mathfrak{g},\lbrace x_i \rbrace_{i\in I}) & (\mathfrak{y},\lbrace a_m \rbrace_{m\in M}) \\
(\mathfrak{g}^{x_k},\lbrace x_i - x_k \rbrace_{i\in I}) & (\mathfrak{y},\lbrace a_m \rbrace_{m\in M}) \\
(\mathfrak{x},\lbrace z_j \rbrace_{j\in J}) & (\mathfrak{y}',\lbrace b_m \rbrace_{m\in M}) ,\\
};
\path[->]
(m-1-1) edge (m-1-2)
(m-1-1) edge (m-2-1)
(m-1-2) edge[-, double] (m-2-2)
(m-2-1) edge (m-2-2)
(m-2-1) edge node[left] {$(\iota , 0)$} (m-3-1)
(m-2-2) edge (m-3-2)
(m-3-1) edge (m-3-2);
\end{tikzpicture}
\end{center}
where the morphism $(\mathfrak{y},\lbrace a_m \rbrace_{m\in M})\to (\mathfrak{y}',\lbrace b_m \rbrace_{m\in M})$ is a strict acyclic fibration and hence is an acyclic fibration in $\mathcal{G}$. Therefore, a lift exists in the bottom square which clearly respects marked points by commutativity of the diagram and hence provides a lift for the outer rectangle implying that the morphism $(\iota,\iota (x_k))$ is a cofibration.
\end{proof}

\begin{proposition}\label{prop_curvedLie_CMC}
$\curvedLie_*$ is a CMC with the model structure of Definition \ref{def_classes}.
\end{proposition}
\begin{proof}
It follows by the preceding results.
\end{proof}

\section{Quillen equivalence}\label{sec_equiv}

This section first recalls the Quillen equivalence of \cite[Section~5]{chuang_laz_mannan} between $\mathcal{E}$ and $\mathcal{G}$. A nice feature of this construction is the contravariant functors interchange curved Lie algebras with non-zero curvature and acyclic algebras, as well as uncurved Lie algebras (or dglas) and Hinich algebras. Using the Quillen equivalence of $\mathcal{E}$ and $\mathcal{G}$ as a foundation, a Quillen equivalence between $\curvedLie_*$ and $\mathcal{A}$ (via $\Prod(\mathcal{E})$) is constructed.

As remarked earlier, when passing from $\mathcal{E}$ to $\mathcal{A}$ (via $\Prod(\mathcal{E})$) it is possible to extend the Quillen equivalence by passing from $\mathcal{G}$ to its category of formal coproducts: c.f.\ \cite{chuang_laz_mannan}. However, this result is unnatural due to its asymmetry. Here---by replacing the category of formal coproducts with the category of curved Lie algebras with curved morphisms---a more symmetric result is obtained.

\begin{definition}\label{def_CE}
A contravariant functor $\CE\colon\mathcal{G}\to\mathcal{E}$ is given by sending a curved Lie algebra $(\mathfrak{g},d,\omega )$ to the local pseudo-compact cdga $\hat{S}\Sigma^{-1} \mathfrak{g}^*$ with differential induced via the Leibniz rule and continuity by its restriction to $\Sigma^{-1} \mathfrak{g}^*$. The restriction is given by the sum of three components:
\begin{itemize}
\item $\Sigma^{-1} \mathfrak{g}^*\to k$ given by the evaluation of the curvature $\omega$ of $\mathfrak{g}$,
\item $\Sigma^{-1} \mathfrak{g}^*\to \Sigma^{-1} \mathfrak{g}^*$ given by pre-composition of the differential $d$, and
\item $\Sigma^{-1} \mathfrak{g}^*\to S^2 \Sigma^{-1} \mathfrak{g}^*$ given by the pre-composition of the bracket on $\mathfrak{g}$.
\end{itemize}
\end{definition}

\begin{remark}
If the curvature of a curved Lie algebra is $0$, then the first part of the differential in Definition \ref{def_CE} disappears recovering the construction of Hinich \cite{hinich_stacks}.
\end{remark}

\begin{proposition}
For any given curved Lie algebra $(\mathfrak{g},d,\omega)\in\mathcal{G}$, the local pseudo-compact cdga $\CE (\mathfrak{g},d,\omega)$ is cofibrant in $\mathcal{E}$.
\end{proposition}
\begin{proof}
Every object of $\mathcal{G}$ is fibrant and $\CE$ maps fibrations to cofibrations.
\end{proof}

\begin{definition}
A contravariant functor $\mathcal{L}\colon\mathcal{E}\to\mathcal{G}$ is given by sending a local pseudo-compact cdga $A$ to the free graded Lie algebra $\operatorname{L}\Sigma^{-1}A^*$ with differential induced by the differential of $A$ and the multiplication of $A$, and curvature given by the morphism $k\to\Sigma^{-1} A^*$ induced by the composition of the augmentation of $A$ with the differential of $A$.
\end{definition}

\begin{remark}
If $A$ is a Hinich algebra the image of the differential on $A$ is its maximal ideal $m_A$ which is precisely the kernel of the augmentation of $A$ and so $\mathcal{L}(A)$ has zero curvature.
\end{remark}

\begin{proposition}
For any given local pseudo-compact cdga, $A$, the curved Lie algebra $\mathcal{L}(A)$ is a cofibrant object in $\mathcal{G}$.
\end{proposition}
\begin{proof}
Every object of $\mathcal{E}$ is fibrant and $\mathcal{L}$ maps fibrations to cofibrations.
\end{proof}

Moving from the categories $\mathcal{G}$ and $\mathcal{E}$ to the categories $\curvedLie_*$ and $\Prod(\mathcal{E})$, the contravariant functors $\mathcal{L}$ and $\CE$ are extended as follows.

\begin{definition}
A contravariant functor $\widetilde{\CE}\colon\curvedLie_*\to \Prod(\mathcal{E})$ is given by
\[
\widetilde{\CE}((\mathfrak{g},d,\omega),\lbrace x_i \rbrace_{i\in I}) := \prod_{i\in I} \CE(\mathfrak{g}, d+ad_{x_i},\omega + dx_i +\frac{1}{2}[x_i,x_i]).
\]
Given a morphism of marked curved Lie algebras
\[
(f,\alpha )\colon(\mathfrak{g},\lbrace x_i \rbrace_{i\in I})\to (\mathfrak{h},\lbrace y_j \rbrace_{j\in J}),
\]
let the $k$th component of the morphism $\widetilde{\CE}(\mathfrak{h},\lbrace y_j \rbrace_{j\in J})\to \widetilde{\CE}(\mathfrak{g},\lbrace x_i \rbrace_{i\in I})$ be given by applying $\CE$ to the rectification $\mathfrak{g}^{x_k}\to\mathfrak{h}^{y_{j_k}}$ considered as a morphism in $\mathcal{G}$.
\end{definition}

\begin{remark}
The contravariant functor $\widetilde{\CE}$ takes a marked Lie algebra to a formal product of Hinich algebras and acyclic algebras depending upon whether the marked point is a MC element or not.
\end{remark}

\begin{definition}
A contravariant functor $\widetilde{\mathcal{L}}\colon\Prod (\mathcal{E})\to\curvedLie_*$ is given by
\[
\widetilde{\mathcal{L}} \left(\prod_{i\in I} A_i \right) :=\coprod_{i\in I} (\mathcal{L}(A_i),\lbrace 0 \rbrace ),
\]
where $\coprod$ is the coproduct of the category $\curvedLie_*$, c.f.\ Proposition \ref{prop_coprod}. Given any morphism $f\colon\prod_{i\in I} A_i \to \prod_{j\in J} B_j$ of $\Prod (\mathcal{E})$, let the morphism
\[
\coprod_{j\in J} (\mathcal{L}(B_j),\lbrace 0\rbrace )\to \coprod_{i \in I} (\mathcal{L}(A_i),\lbrace 0\rbrace )
\]
be obtained by combining each of the components
\[
(\mathcal{L}(f_j),0)\colon (\mathcal{L}(B_j),\lbrace 0 \rbrace )\to (\mathcal{L}(A_{i_j}),\lbrace 0 \rbrace ).
\]
\end{definition}

\begin{proposition}
The contravariant functor $\widetilde{\mathcal{L}}$ is left adjoint to the contravariant functor $\widetilde{\CE}$.
\end{proposition}
\begin{proof}
The statement of the proposition is equivalent to the statement there exists an isomorphism of the sets
\[
\Hom_{\Prod(\mathcal{E})} \left(\widetilde{\CE}(\mathfrak{g},\lbrace x_i \rbrace_{i\in I}),\prod_{j\in J} A_j \right)\cong\Hom_{\curvedLie_*} \left(\widetilde{\mathcal{L}}\left(\prod_{j\in J} A_j \right), (\mathfrak{g},\lbrace x_i \rbrace_{i\in I})\right).
\]
Given a morphism $f\colon\prod_{i\in I} \CE (\mathfrak{g}^{x_i}) \to \prod_{j\in J} A_j$ of $\Prod (\mathcal{E})$, applying the contravariant functor $\mathcal{L}$ to the components, $f_j$, the morphisms 
\[
\mathcal{L}(f_j)\colon \mathcal{L}(A_j)\to \mathcal{L}(\CE(\mathfrak{g}^{x_{i_j}}))
\]
are obtained. Since the contravariant functors $\mathcal{L}$ and $\CE$ are adjoint, these morphisms are equivalent to the morphisms $\mathcal{L}(f_j)\colon\mathcal{L} (A_j)\to \mathfrak{g}^{x_{i_j}}$. Therefore, by the universal property of coproducts, for each $j\in J$ there exists the commutative diagram
\begin{center}
\begin{tikzpicture}
\matrix (m) [matrix of math nodes, row sep=2em, column sep=2em]{
 & \mathfrak{g} \\
\mathcal{L}(A_j) & \coprod_{j\in J} \mathcal{L}(A_j ). \\};
\path[->]
(m-2-1) edge node[auto] {$(id,-x_{i_j})\circ \mathcal{L}(f_j)$} (m-1-2)
(m-2-2) edge (m-1-2)
(m-2-1) edge (m-2-2);
\end{tikzpicture}
\end{center}
Clearly, the morphisms are marked and therefore it has been shown that $\widetilde{\mathcal{L}}(f)$ is a morphism in $\Hom_{\curvedLie_*} (\widetilde{\mathcal{L}}(\prod_{j\in J} A_j ), (\mathfrak{g},\lbrace x_i \rbrace_{i\in I}))$.

Conversely, given a morphism $f\colon\widetilde{\mathcal{L}}(\prod_{j\in J} A_j)\to(\mathfrak{g},\lbrace x_i \rbrace_{i\in I})$ one can easily use the fact that $\CE$ and $\mathcal{L}$ are an adjoint pair to show---in a similar manner to the preceding---that after applying the contravariant functor $\widetilde{\CE}$, a morphism equivalent to \[
\prod_{i\in I}\CE(\mathfrak{g}^{x_i})\to\prod_{j\in J} A_j
\]
is obtained.
\end{proof}

\begin{proposition}\label{prop_L_cofibrant}
The contravariant functors $\widetilde{\mathcal{L}}$ and $\widetilde{\CE}$ both map fibrations to cofibrations.
\end{proposition}
\begin{proof}
Given a fibration $f\colon\prod_{i\in I} A_i \to \prod_{j\in J} B_j$ of $\Prod(\mathcal{E})$, for each $i\in I$ the morphism $f^i \colon A_i\to \prod_{\lbrace j : j\mapsto i \rbrace} B_j$ is  a fibration of $\mathcal{E}$. Hence for each $i \in I$ the morphism $\mathcal{L}(f^i)$ is a cofibration. Let
\begin{center}
\begin{tikzpicture}
\matrix (m) [matrix of math nodes, row sep=2em, column sep=2em]{
\left(\mathcal{L}\left(\prod_{\lbrace j : j\mapsto i \rbrace} B_j\right),\lbrace 0 \rbrace\right) & \coprod_{j\in J} (\mathcal{L}(B_j),\lbrace 0 \rbrace) & (X,\lbrace \xi_k \rbrace_{k\in K} )\\
(\mathcal{L}(A_i),\lbrace 0 \rbrace) & \coprod_{i\in I} (\mathcal{L}(A_i),\lbrace 0 \rbrace) & (Y,\lbrace\nu_k\rbrace_{k\in K}), \\};
\path[->]
(m-1-1) edge (m-1-2)
(m-1-1) edge (m-2-1)
(m-1-2) edge (m-1-3)
(m-1-2) edge (m-2-2)
(m-2-1) edge (m-2-2)
(m-1-3) edge (m-2-3)
(m-2-2) edge (m-2-3);
\end{tikzpicture}
\end{center}
be a commutative diagram in $\curvedLie_*$ and $(X,\lbrace \xi_k \rbrace_{k\in K)})\to(Y,\lbrace\nu_k\rbrace_{k\in K})$ be an acyclic fibration. Further, it is assumed to be strict, c.f.\ Proposition \ref{prop_rect_square}. The left hand morphism is a cofibration of $\mathcal{G}$ and as such there exists some lift $(\mathcal{L}(A_i),\lbrace 0 \rbrace )\to (X,\lbrace \xi_k \rbrace_{k\in K})$ for each $i\in I$. Using the universal property of coproducts, there exists a morphism $\coprod_{i\in I} (\mathcal{L}(A_i),\lbrace 0 \rbrace )\to (X,\lbrace \xi_k \rbrace_{k\in K})$ making everything commute. Hence, the morphism $\widetilde{\mathcal{L}}(f)$ is a cofibration. The proof for $\widetilde{\CE}$ is similar.
\end{proof}

\begin{proposition}
The contravariant functors $\widetilde{\CE}$ and $\widetilde{\mathcal{L}}$ preserve weak equivalences.
\end{proposition}
\begin{proof}
Given a weak equivalence $(f,\alpha)\colon(\mathfrak{g},\lbrace x_i \rbrace_{i\in I})\to (\mathfrak{h},\lbrace y_i \rbrace_{i\in I})$ of $\curvedLie_*$, it is evident by definition that for all $i\in I$ the morphism $\CE(\mathfrak{h}^{y_i})\to \CE(\mathfrak{g}^{x_i})$ is a weak equivalence of $\mathcal{E}$, whence $\widetilde{\CE}(f,\alpha)$ is a weak equivalence.

Suppose $f\colon\prod_{i\in I} A_i \to \prod_{i\in I} B_i$ is a weak equivalence of $\Prod(\mathcal{E})$. Clearly $\widetilde{\mathcal{L}}(f)$ induces a bijection upon marked points. Now, $f$ has components $f_i\colon A_i\to B_i$ which are weak equivalences in $\mathcal{E}$, hence $(\mathcal{L}(f_i),0)$ is a weak equivalence for each $i$. Since $\widetilde{\mathcal{L}}(A)$ is cofibrant for all $A$ (by Proposition \ref{prop_L_cofibrant}) one can conclude the coproducts $\coprod_i \widetilde{\mathcal{L}}(A_i)$ and $\coprod_i \widetilde{\mathcal{L}}(B_i)$ descend to the level of homotopy, where each morphism $(\mathcal{L}(f_i),0)$ is an isomorphism, which implies $\widetilde{\mathcal{L}}(f)$ is an isomorphism and so $\widetilde{\mathcal{L}}(f)$ is a weak equivalence.
\end{proof}

\begin{theorem}\label{thm_main}
The categories $\Prod(\mathcal{E})$ and $\curvedLie_*$ are Quillen equivalent.
\end{theorem}
\begin{proof}
One applies the preceding results with \cite[Theorem 9.7]{dwyer_spalinski}.
\end{proof}

\section{Deformation functors over pseudo-compact cdgas}\label{sec_deformation}

An application of the above constructions to algebraic deformation theory is contained within this section; more precisely, the theory is extended to deformation functors over (not necessarily local) pseudo-compact cdgas. The main result, Theorem \ref{thm_rep_def}, shows these deformation functors are representable in $\Ho(\mathcal{A})$. The story of the approach to deformation theory via dglas has its roots with Drinfeld \cite{drinfeld} and Kontsevich \cite{kontsevich_notes} among others, who noticed that deformation theories in characteristic zero are governed by the $\MC$ elements of dglas.

Here the category of sets is denoted by $\mathscr{S}$ and the notation $(A,m_A)$ refers to a local pseudo-compact cdga $A$ with maximal ideal given by $m_A$. 

\subsection{MC elements and local pseudo-compact cdga morphisms}

Fixing a curved Lie algebra $(\mathfrak{g},d_\mathfrak{g},\omega_\mathfrak{g})$ and a local pseudo-compact cdga, $(A,m_A)$, recall that the tensor product $(\mathfrak{g}\hat{\otimes}A,d,\omega_\mathfrak{g}\hat{\otimes} 1)$ is a well defined curved Lie algebra. The $\MC$ elements of the tensor product $\mathfrak{g}\hat{\otimes} A$ can be considered, in the usual sense, as those elements, $x\in \mathfrak{g}\hat{\otimes}A$, solving the $\MC$ equation: $\omega_\mathfrak{g} \hat{\otimes} 1 + dx +\frac{1}{2}[x,x]= 0$. However, for the construction of the deformation functor given later in this section, only the subset of those $\MC$ elements belonging to $\mathfrak{g}\hat{\otimes}m_A$ are examined.

\begin{definition}\label{def_MC_functor}
Let $\widetilde{\MC} (\mathfrak{g}\hat{\otimes} A)$ denote the set of $\MC$ elements belonging to the subset $\mathfrak{g}\hat{\otimes}m_A$.
\end{definition}

In fact $\widetilde{\MC}$ can be viewed as a bifunctorial construction from the product category of curved Lie algebras with local pseudo-compact cdga into the category of sets.

\begin{proposition}\label{prop_iso_func2}
The functors $(\mathfrak{g},A)\mapsto \Hom (\CE(\mathfrak{g}),A))$ and $(\mathfrak{g},A)\mapsto \widetilde{\MC}(\mathfrak{g}\hat{\otimes} A)$ are naturally isomorphic.
\end{proposition}
\begin{proof}
A degree $-1$ element in $\mathfrak{g}\hat{\otimes} m_A$ is a degree $0$ element in $\Sigma^{-1} \mathfrak{g}\hat{\otimes} m_A$, further this element determines (and is determined by) a continuous linear morphism $(\Sigma^{-1} \mathfrak{g})^* \to m_A$. In turn, this continuous linear morphism determines (and is determined) by a morphism of local pseudo-compact commutative graded algebras: $\CE (\mathfrak{g})\to A$. The condition that this morphism commutes with the differential is precisely the condition that the element belongs to $\widetilde{\MC} (\mathfrak{g}\hat{\otimes} A)$.
\end{proof}

The equivalence of functors in Proposition \ref{prop_iso_func2} motivates the notion of a homotopy between elements of the set $\widetilde{\MC}(\mathfrak{g}\hat{\otimes} A)$.

\begin{definition}\label{def_algebra_path_object}\hfill
\begin{itemize}
\item Let $k[z,dz]$ denote the free unital cdga on the generators $z$ and $dz$ with degrees $0$ and $-1$ respectively, with the differential given by $d(z)=dz$.
\item Given $A\in\mathcal{E}$, let $A[z,dz]$ be the cdga given by $A\hat{\otimes} k[z,dz]$, and let the quotient morphisms given by setting $z$ to $0$ or $1$ be denoted $|_0,|_1 \colon A[z,dz]\to A$, respectively.
\end{itemize}
\end{definition}

\begin{remark}
$k[z,dz]$ is the familiar de Rham algebra of forms on the unit interval.
\end{remark}

\begin{definition}
Two elements $\xi,\eta \in \widetilde{\MC} (\mathfrak{g}\hat{\otimes} A)$ are said to be homotopic if there exists an element $h\in \widetilde{\MC} (\mathfrak{g}\hat{\otimes} A[z,dz])$ such that $h|_0=\xi$ and $h|_1=\eta$.
\end{definition}

Using the isomorphism of functors in Proposition \ref{prop_iso_func2}, $h\in \widetilde{\MC}(\mathfrak{g}\hat{\otimes} A[z,dz])$ corresponds to a morphism of pseudo-compact cdgas $h\colon\CE(\mathfrak{g})\to A[z,dz]$ which when specialising to $z=0$ and $z=1$ restricts to the morphisms corresponding to $\xi,\eta\in \widetilde{\MC} (\mathfrak{g}\hat{\otimes} A)$, respectively. Thus, a homotopy of $\MC$ elements is nothing more than a Sullivan homotopy between the corresponding morphisms of local pseudo-compact cdgas, c.f.\ Appendix \ref{sec_sullivan}.

\begin{definition}
Let $\widetilde{\mathcal{MC}}(\mathfrak{g}\hat{\otimes}A)$ denote the set of equivalence classes of $\widetilde{\MC} (\mathfrak{g}\hat{\otimes} A)$ modulo the homotopy relation. This set is called the Maurer-Cartan moduli set of the curved Lie algebra $\mathfrak{g}$ with coefficients in the local pseudo-compact cdga $A$.
\end{definition}

\begin{remark}
The $\MC$ moduli set can be obtained in several ways in slightly specialised cases. For example, it can be seen as the set of connected components of the $\MC$ simplicial set; see \cite{laz_markl}, it should be noted, though, that pseudo-compact dglas are considered in op. cit. Additionally, a result of Schlessinger and Stasheff \cite{schlessinger_stasheff_deformation_rational} states that for a pronilpotent dgla, two $\MC$ elements are homotopic if, and only if, they are gauge equivalent: see \cite{chuang_laz_feynman} for a discussion and a proof.
\end{remark}

The homotopy relation of MC elements and the Sullivan homotopy of morphisms of pseudo-compact local cdgas are so closely related that Proposition \ref{prop_iso_func2} also holds on the level of homotopy.

\begin{proposition}\label{prop_isofunc_homotopy}
Given a curved Lie algebra $\mathfrak{g}$ and a local pseudo-compact cdga $(A,m_A)$, the functors
\[
(\mathfrak{g},A)\mapsto \Hom_{\operatorname{Ho}(\mathcal{E})} (\CE(\mathfrak{g}),A)) \;\; \mathrm{and} \;\; (\mathfrak{g},A)\mapsto \widetilde{\mathcal{MC}}(\mathfrak{g}\hat{\otimes} A)
\]
are naturally isomorphic.
\end{proposition}
\begin{proof}
The two definitions of homotopy are equivalent since $\CE(\mathfrak{g})$ is cofibrant (see Theorem \ref{thm_sullivan_homotopy}), and so the result follows.
\end{proof}

\subsection{Deformations over pseudo-compact cdgas}

For a dgla $\mathfrak{g}$ and a local Artin algebra $(A,m_A)$, recall the tensor $\mathfrak{g}\hat{\otimes} m_A$ possesses a well-defined dgla structure. Shadowing the inspiration of Drinfeld, Kontsevich, et al., where the deformation functor associated to $\mathfrak{g}$ is defined as the functor mapping $(A, m_A)$, to the set of $\MC$ elements of $\mathfrak{g}\hat{\otimes} m_A$ modulo gauge equivalence, the following definition is made.

\begin{definition}\label{def_def_functor}
Fixing a marked curved Lie algebra $(\mathfrak{g},\lbrace x_i \rbrace_{i\in I})\in \curvedLie_*$, a deformation functor $\operatorname{Def}_{(\mathfrak{g},\lbrace x \rbrace_{i\in I} )}\colon\mathscr{A} \to\mathscr{S}$ is given by
\[
A=\prod_{j\in J} A_j \mapsto \prod_{j\in J} \coprod_{i\in I} \widetilde{\mathcal{MC}}\left( \mathfrak{g}^{x_i} \hat{\otimes} A_j\right).
\]
\end{definition}

Using Proposition \ref{prop_algebra_morphisms} and Proposition \ref{prop_isofunc_homotopy} one arrives at the following theorem.

\begin{theorem}\label{thm_rep_def}
The deformation functor $\operatorname{Def}_{(\mathfrak{g},\lbrace x_i \rbrace_{i\in I})} \colon\mathscr{A} \to \mathscr{S}$ is representable in the homotopy category of $\mathcal{A}$ by $\widetilde{\CE} (\mathfrak{g},\lbrace x_i \rbrace_{i\in I})$.
\end{theorem}
\begin{proof}
Let $A\in\mathcal{A}$,
\[
\begin{aligned}
\operatorname{Def}_{(\mathfrak{g},\lbrace x_i \rbrace_{i\in I})} (A)&=
\prod_{j\in J} \coprod_{i\in I} \widetilde{\mathcal{MC}}\left( \mathfrak{g}^{x_i} \hat{\otimes} A_j\right) \\
& \cong \prod_{j\in J} \coprod_{i\in I} \Hom_{\Ho (\mathcal{E})} \left( \CE (\mathfrak{g}^{x_i}), A_j \right) \\
& \cong \Hom_{\Ho (\mathcal{A})} \left(\widetilde{\CE} (\mathfrak{g}, \lbrace x_i \rbrace_{i\in I}), A\right). \\
\end{aligned}
\]
\end{proof}

\begin{corollary}
The functor $\operatorname{Def}_{(\mathfrak{g},\lbrace x_i \rbrace_{i\in I})}$ is homotopy invariant in both the marked Lie algebra and pseudo-compact cdga variables.
\end{corollary}
\begin{proof}
$\widetilde{\CE}$ is homotopy invariant and $\operatorname{Def}_{(\mathfrak{g},\lbrace x_i \rbrace_{i\in I})}$ is representable in the homotopy category of pseudo-compact cdgas, thus the result follows.
\end{proof}

The functor given in Definition \ref{def_def_functor} relies upon the decomposition of a pseudo-compact cdga into the product of local ones. By recalling a pair of Quillen adjoint functors originally given in \cite{chuang_laz_mannan} a more satisfying equivalent definition that does not rely upon a decomposition can be given in the case where the curved Lie algebra possesses one marked point. Let the functor $F\colon\mathcal{E}\to\Prod (\mathcal{E})$ be given by the inclusion of a local pseudo-compact cdga to the formal product over a singleton set, and let $G\colon\Prod (\mathcal{E})\to\mathcal{E}$ be given by taking the formal product to the actual product of $\mathcal{E}$. It is proven in \cite{chuang_laz_mannan} these functors are Quillen adjoint. As defined here, $G$ uses the decomposition of a pseudo-compact cdga into the product of local ones. However, it is possible to define $G$ without this luxury, c.f.\ \cite{chuang_laz_mannan}.

\begin{definition}\label{def_def_functor2}
Fixing a marked curved Lie algebra $(\mathfrak{g},\lbrace x \rbrace)\in \curvedLie_*$, a deformation functor $\widehat{\operatorname{Def}}_{(\mathfrak{g},\lbrace x \rbrace )}\colon\mathscr{A} \to\mathscr{S}$ is given by $A\mapsto \widetilde{\mathcal{MC}} \left(\mathfrak{g}^x \hat{\otimes}\;G (A)\right)$.
\end{definition}

\begin{proposition}\label{prop_iso_def_functors}
Fix a marked curved Lie algebra $(\mathfrak{g},\lbrace x \rbrace)\in\curvedLie_*$. The functors $\widehat{\operatorname{Def}}_{(\mathfrak{g},\lbrace x \rbrace )}$ and $\operatorname{Def}_{(\mathfrak{g},\lbrace x \rbrace )}$ are naturally isomorphic.
\end{proposition}
\begin{proof}
It is immediate from the definitions.
\end{proof}

\begin{theorem}
The functor $\widehat{\operatorname{Def}}_{(\mathfrak{g},\lbrace x \rbrace )}$ is representable in the homotopy category of pseudo-compact cdgas by $\widetilde{\CE}(\mathfrak{g},\lbrace x \rbrace )$.
\end{theorem}
\begin{proof}
It follows from Proposition \ref{prop_iso_def_functors}.
\end{proof}

\appendix

\section{Sullivan homotopy and path objects}\label{sec_sullivan}

In Section \ref{sec_deformation} the notion of a Sullivan homotopy is used to relate homotopy classes of MC elements in certain curved Lie algebras with the homotopy classes of morphisms in $\mathcal{E}$. A Sullivan homotopy is reminiscent of a right homotopy in a CMC, but in $\mathcal{E}$ this is not quite the case: the candidate for a so-called path object is not pseudo-compact. One could attempt to fix this failure by extending $\mathcal{E}$ suitably. This approach, however, is not be taken here as it is not necessary.

Within this section interest is restricted to the categories $\mathcal{E}$ and $\mathcal{G}$; most importantly, all morphisms between curved Lie algebras are strict.

\begin{definition}\label{def_path_objects}
Given $\mathfrak{g}\in\mathcal{G}$, let $\mathfrak{g}[z,dz]\in\mathcal{G}$ be given by $\mathfrak{g}\otimes k[z,dz]$, and the quotient morphisms given by setting $z$ to $0$ and $1$ are denoted $|_0,|_1 \colon\mathfrak{g}[z,dz]\to \mathfrak{g}$, resp.
\end{definition}

Recall Definition \ref{def_algebra_path_object}. As already remarked, the objects $A[z,dz]$ and $\mathfrak{g}[z,dz]$ resemble path objects for $A$ and $\mathfrak{g}$, respectively. However, $k[z,dz]$ is not pseudo-compact, and $A[z,dz]\notin\mathcal{E}$. Therefore, $A[z,dz]$ cannot be a path object for $A$ in $\mathcal{E}$.

\begin{proposition}
Given $\mathfrak{g}\in\mathcal{G}$, $\mathfrak{g}[z,dz]$ is a very good path object for $\mathfrak{g}$ in $\mathcal{G}$.
\end{proposition}
\begin{proof}
When $\mathfrak{g}$ has zero curvature (i.e.\ a dgla) the statement is already known. Assuming $\mathfrak{g}$ has non-zero curvature, the canonical morphism $\mathfrak{g}\to\mathfrak{g}[z,dz]$ is between two curved Lie algebras with non-zero curvature and, therefore, a weak equivalence. Moreover, the diagonal morphism $(id_\mathfrak{g},id_\mathfrak{g})\colon\mathfrak{g}\to\mathfrak{g}\prod\mathfrak{g}$ can be factorised as $\mathfrak{g}\xrightarrow{\sim}\mathfrak{g}[z,dz]\to\mathfrak{g}\prod\mathfrak{g}$ with the morphism $\mathfrak{g}[z,dz]\to\mathfrak{g}\prod\mathfrak{g}$ clearly surjective and hence a fibration. Further, the morphism $\mathfrak{g}\to\mathfrak{g}[z,dz]$ can easily be shown to be a cofibration by showing it has the LLP with respect to all acyclic fibrations.
\end{proof}

\begin{definition}\hfill
\begin{itemize}
\item Let $A, B\in\mathcal{E}$. Two parallel morphisms of $\mathcal{E}$, $f,g\colon A\to B$, are said to be Sullivan homotopic if there exists a continuous local cdga morphism $h\colon A\to B[z,dz]$ such that $h|_0=f$ and $h|_1=g$.
\item Let $\mathfrak{g},\mathfrak{h}\in\mathcal{G}$. Two parallel morphisms of $\mathcal{G}$, $f,g\colon\mathfrak{g}\to \mathfrak{h}$, are said to be Sullivan homotopic if there exists a curved Lie algebra morphism $h\colon\mathfrak{g}\to \mathfrak{h}[z,dz]$ such that $h|_0=f$ and $h|_1=g$.
\end{itemize}
\end{definition}

Since all objects of $\mathcal{G}$ are fibrant, the right homotopy with the path object given above is an equivalence relation, c.f.\ \cite[Lemma 4.16]{dwyer_spalinski}. Therefore a Sullivan homotopy in $\mathcal{G}$ is simply a right homotopy. When the source object is cofibrant this notion of homotopy coincides with that given by the weak equivalences of the category $\mathcal{G}$, whence the notion of a Sullivan homotopy coincides with that of the model structure.

It was proven in \cite[Theorem 3.6]{andrey_MC} that for a cofibrant Hinich algebra, $A$, and an arbitrary Hinich algebra, $B$, the set of equivalence classes of Sullivan homotopic morphisms $A\to B$ is in bijection with the set of morphisms $A\to B$ in the homotopy category of the Hinich category. This theorem extends to $\mathcal{E}$.

\begin{theorem}[Lazarev]\label{thm_sullivan_homotopy}
Given a cofibrant $A\in\mathcal{E}$ and any arbitrary $B\in\mathcal{E}$, two parallel morphisms $f,g\colon A\to B$ are Sullivan homotopic if, and only if, they represent the same morphism in $\Ho (\mathcal{E})$. More precisely, there is a bijective correspondence between equivalence classes of Sullivan homotopic morphisms $A\to B$ in $\mathcal{E}$ and the set of morphisms $A\to B$ in $\Ho (\mathcal{E})$.
\end{theorem}
\begin{proof}
The proof follows, mutatis mutandis, from \cite[Theorem 3.6]{andrey_MC}.
\end{proof}

It should be mentioned, albeit in the conclusion of this appendix, Pridham \cite{pridham} constructed a path object functor for Hinich algebras. This construction is briefly recalled here. Given a finite-dimensional nilpotent local cdga whose maximal ideal respects the differential, $A$, consider the pullback
\begin{center}
\begin{tikzpicture}
\matrix (m) [matrix of math nodes, row sep=2em, column sep=2em]{
A[z,dz]\times_{k[z,dz]} k & k \\
A[z,dz] & k[z,dz], \\
};
\path[->]
(m-1-1) edge (m-1-2)
(m-1-1) edge (m-2-1)
(m-1-2) edge[right hook->] (m-2-2)
(m-2-1) edge node[below] {$\varepsilon$} (m-2-2);
\end{tikzpicture}
\end{center}
where $\varepsilon$ is the induced augmentation. This pullback is isomorphic to the algebra $m_A [z,dz]\oplus k$. The path object of $A$ is given by the completion of $m_A [z,dz]\oplus k$ with respect to ideal $m_A [z,dz]$. This construction is extended to the category of Hinich algebras by using the exactness of the completion functor for finitely generated algebras. For more details on this path object construction see Pridham's paper \cite{pridham} where, using this path object functor, a cylinder object functor is also constructed.

\bibliographystyle{plain}
\bibliography{my_bib}

\end{document}